\documentclass[a4paper,11pt]{amsart}
\usepackage[T1]{fontenc}
\usepackage[utf8]{inputenc}
\usepackage{lmodern}

\usepackage{amsmath, amsthm,amssymb, color, comment,graphicx}
\theoremstyle{definition} %%% for statements in roman typeface

 \newtheorem{definition}{Definition}[section]
  \newtheorem{coro}{Corollary}[section]
 \newtheorem{remark}[definition]{Remark}
 \newtheorem{example}[definition]{Example}
 
\theoremstyle{plain}      %%% for statements in italic typeface

 \newtheorem{proposition}[definition]{Proposition}
 \newtheorem{theorem}[definition]{Theorem}
 
 \newtheorem{lemma}[definition]{Lemma}

\newtheorem*{theorem*}{Theorem}

\newcommand{\R}{\mathbb{R}}
\newcommand{\C}{\mathbb{C}}
\newcommand{\Z}{\mathbb{Z}}
\newcommand{\Q}{\mathbb{Q}}
\renewcommand{\P}{\mathbb{P}}
\newcommand{\SL}{\mathrm{SL}}
\renewcommand{\epsilon}{\varepsilon}
\DeclareMathOperator{\re}{Re}
\DeclareMathOperator{\im}{Im}
\DeclareMathOperator{\card}{Card}
\DeclareMathOperator{\inter}{Int}
\DeclareMathOperator{\homo}{Hom}
\DeclareMathOperator{\ord}{ord}
\DeclareMathOperator{\fraction}{Frac}
\title{Volume function and Mahler measure of exact polynomials}
\date{}
\author{Antonin Guilloux} 
\address{Sorbonne Universit\'e, IMJ-PRG, 75252 Paris c\'edex 05, France}
\email{antonin.guilloux@imj-prg.fr}

\author{Julien March\'e}
\address{Sorbonne Universit\'e, IMJ-PRG, 75252 Paris c\'edex 05, France}
\email{julien.marche@imj-prg.fr}

\begin{document}
%\tableofcontents

\begin{abstract}
We study a class of 2-variable polynomials called exact polynomials which contains $A$-polynomials of knot complements. The Mahler measure of these polynomials can be computed in terms of a volume function defined on the vanishing set of the polynomial. We prove that the local extrema of the volume function are on the 2-dimensional torus and give a closed formula for the Mahler measure in terms of these extremal values. This formula shows that the Mahler measure of an irreducible and exact polynomial divided by $\pi$ is greater than the amplitude of the volume function. We also prove a $K$-theoretical criterium for a polynomial to be a factor of an $A$-polynomial and give a topological interpretation of its Mahler measure. 
\end{abstract}

\maketitle

\section*{Introduction}

A polynomial $P\in \C[X^{\pm 1}, Y^{\pm 1}]$ vanishing on a curve $C\subset (\C^*)^2$ is said to be {\it exact} if there exists a function $V:C\to \R$ (called volume function) satisfying 
$$dV=\log |y| d\arg x-\log |x| d\arg y.$$
In this article, we study the properties of volume functions $V$. For instance we show that the volume function extends continuously to the smooth projective model $\hat{C}$ of $C$ and study its local extrema. Our main result concerning $V$, proved in Section 2 is that the extrema of a volume function are only attained at so-called \emph{toric points}:
\begin{theorem*}
The local extrema of $V$ are necessarily finite points of $\hat{C}$ projecting to pairs $(x,y)\in(\C^*)^2$ satisfying $|x|=|y|=1$. 
\end{theorem*}
The proof is rather elementary and can be visualised with the help of two notions from real algebraic geometry: the logarithmic Gauss map and amoebas. 

Our motivation comes from topology: let $M$ be a 3-manifold with toric boundary, we denote by $X(M)$ its character variety, that is the algebraic quotient $\homo(\pi_1(M),\SL_2(\C))/\!/\SL_2(\C)$. 
The character variety of the boundary $X(\partial M)$ is the quotient of $(\C^*)^2$ by the involution $(x,y)\sim(x^{-1},y^{-1})$ and the restriction map $r:X(M)\to X(\partial M)$ has the property that $\overline{r(X(M))}$ is the vanishing set of some polynomial $A_M\in \Z[X^{\pm 1},Y^{\pm 1}]$. This polynomial - called the $A$-polynomial of $M$ - was introduced by \cite{CCGLS94} and is known to be exact. 

Indeed, given a representation $\rho:\pi_1(M)\to\SL_2(\C)$, $V(r([\rho]))$ is the volume of the representation $\rho$, for which we refer to \cite{francaviglia}. Hence our first theorem partially recovers a recent result of Francaviglia and Savini \cite{fs17}. They prove that the volume function $V:X(M)\to\R$ cannot reach its maximum at ideal points. However our simpler proof works only under the assumption that the restriction map $r:X(M)\to X(\partial M)$ is proper. This assumption holds for instance if $M$ does not contain any closed incompressible surface. These considerations were a starting point for these work but we will not go further in this direction.

Instead, we investigate in Section 3 the computation of the Mahler measure of an exact polynomial. Given a 2-variable polynomial $P\in \C[X^{\pm 1},Y^{\pm 1}]$, the problem of computing explicitly its logarithmic Mahler measure 
\[m(P)=\frac{1}{(2\pi)^2}\int_0^{2\pi}\int_0^{2\pi}\log|P(e^{i\theta},e^{i\phi})|d\theta d\phi.\]
looked intractable before the remarkable computation by Smyth of the Mahler measure of the 2-variable polynomial $X+Y-1$ \cite{Smyth}. Since then, many new examples have been found. For instance, in the article \cite{BoydRodriguezVillegas} building on \cite{BRV-I, Boyd-invariants}, the authors used K-theoretic tools to exhibit a class of 2-variables polynomials with the property that their Mahler measure multiplied by $\pi$ is a rational linear combination of evaluations of the Bloch-Wigner dilogarithm at algebraic arguments. They proceed to give a number theoretic interpretation
of this sum of dilogarithms. They give a lot of examples, among them all $A$-polynomials of $3$-manifolds $M$ with toric boundary. Finally, they observed that the Mahler measure multiplied by $\pi$ is often -- but not always -- equal to the hyperbolic volume of $M$.
Other techniques, which seem unrelated to our work, allow the computation of the Mahler measure for non-exact polynomials, e.g. \cite{Lalin, BrunaultNeururer, Maillot}. 
We refer to the survey \cite{BertinLalin} for a description of these works.

Starting from the computation of the Mahler measure of an exact polynomial (borrowed from \cite{BoydRodriguezVillegas}), it is known that the formula only involves the values of $V$ at critical points which sit inside the torus. What was not known is that the contribution of each critical point can be computed directly. Here is a simple version with strong assumptions
granting that only the local extrema appear:

\begin{theorem*}
{Let $P\in\C[X^{\pm 1},Y^{\pm 1}]$ be an irreducible exact polynomial vanishing on $C$ with volume function $V$. 
Suppose that the curve $C$ is transversal to the torus $S^1\times S^1$ in $(\C^*)^2$. Then up to normalizing $P$ conveniently one has
$$2\pi m(P)=\sum_i V(M_i)-\sum_j V(m_j)$$
where the $M_i$s and the $m_j$s are respectively the local maxima and minima of $V$.}
\end{theorem*} 
We provide a version with weaker transversality assumptions (Theorem \ref{formule}) and a general formula (Theorem \ref{formule-generale}). {This proves that the Mahler mesure is greater than the amplitude of the volume function (see Theorems \ref{inegalite} and \ref{inegalite-generale}):}
\begin{theorem*}
{Let $P$ be a (suitably normalized) irreducible exact polynomial and $V$ a volume function. Then we have:
$$2\pi m(P)\ge \max V- \min V.$$} 
\end{theorem*}
This inequality is particularly nice in the context of hyperbolic manifolds as the maximum of our volume function on $X(M)$ is the hyperbolic volume of the 3-manifold $M$ denoted by vol$(M)$. 
When $A_M$ is irreducible over $\C$, we get the inequality 
$$ \pi m(A_M)\ge {\rm vol}(M).$$
This sheds some light on the cases were equality were observed; manifolds with increasing complexity satisfy a strict inequality and we will give an example of this phenomenon. 

Although it is easy to understand which polynomials are exact in genus 0, the problem looks intractable for higher genus cases. For example, the question found in the fourth final remark in \cite[Section 8]{BoydRodriguezVillegas} reads:
no continuous family of exact polynomials exists. In Section 4, we obtain for genus 1 polynomials the following finiteness result which we expect to be true without assumptions on the genus. 
 \begin{theorem*}
Up to monomial transformations, there is a finite number of exact polynomials $P\in \overline{\Q}[X^{\pm 1},Y^{\pm 1}]$ of genus $g\le 1$ with Newton polygon of bounded area. 
\end{theorem*}

The last part of the article contains some $K$-theoretic arguments. We prove the following theorem. 
\begin{theorem*}Let $P\in \overline{\Q}[X^{\pm 1},Y^{\pm 1}]$ be a polynomial. It satisfies the following condition 
$$\{X,Y\}=0\in K_2(K_P)\otimes \Q\textrm{ where }K_P=\fraction\left( \overline{\Q}[X^{\pm 1},Y^{\pm 1}]/P\right)$$
if and only if $P$ is a factor of the $A$-polynomial of some 3-manifold with boundary. 
\end{theorem*}
The $K$-theoretic condition above is the same as in \cite{BoydRodriguezVillegas}. The proof given in Subsection \ref{ss:Ktheory} borrows arguments from Ghys \cite{ghys}. Hence, at least theoretically, one can recognise which polynomials are $A$-polynomials of 3-manifolds although the criterium is effective only for polynomials defining a curve of genus 0. We then describe the computation of Mahler measures of a few A-polynomials, recovering in part known results.

We end the section and the article by proving a formula which gives a topological interpretation 
of $m(A_M)$ where $A_M$ is the $A$-polynomial of $M$. 
Given a closed manifold $M$ with a knot $K$ inside, we set 
$$m(M,K)=\sum_{[\rho]\in X(M)}\log ||\rho(K)||$$
where $||A||$ is the spectral radius of $A$. If $\rho$ is the hyperbolic representation, 
$\log ||\rho(K)||$ is the length of the geodesic represented by $K$. Hence $m(M,K)$ is 
the sum of all ``lengths" of $K$ over all possible (non necessarily geometric) representations. 
Our result is the following: 

\begin{theorem*}Let $M$ be a manifold with toric boundary satisfying the hypothesis of Proposition \ref{topinterpretation}. We have 
$$\lim_{p^2+q^2\to \infty} m(M_{p/q},K_{p/q})=m(A_M)$$
where $M_{p/q}$ denotes the Dehn filling of $M$ with parameters $p/q$ and $K_{p/q}$ is the core of the surgery. 
\end{theorem*}

{\bf Acknowledgements:} We would like to thank N. Bergeron, G. Calsamiglia, G. Courtois, M. Culler, B. Deroin, E. Falbel, E. Ghys, J. Josi for valuable discussions around this article. 

\section{Exact polynomials, volume function}

Consider a Laurent polynomial $P\in\C[X^{\pm 1},Y^{\pm 1}]$ with two variables 
and let $C$ be the set of its complex smooth points, that is:
\[ C=\{(x,y)\in\C^*,\, P(x,y)=0, dP(x,y)\ne 0\}\]

We define the 2-form $\eta$ on $(\C^*)^2$ by the formula 
\[ \eta= \log|y| d \arg(x)-\log|x| d\arg(y)\]
This form restricted to $C$ is closed as one has 
$d\eta=-\im \left(\frac{dx}{x}\wedge\frac{dy}{y}\right)$.
Note that, comparing with \cite{BoydRodriguezVillegas}, their $\eta$ is
minus ours. This different normalization is mainly due to simplifications of
notations. 

\begin{definition}
We will say that $P\in\C[X^{\pm 1},Y^{\pm 1}]$ is {\it exact} if 
the form $\eta$ restricted to $C$ is exact. 
\end{definition}

\begin{definition}
A {\it volume function} associated to an exact polynomial $P\in \C[X^{\pm 1},Y^{\pm 1}]$ is any 
function $V:C\to\R$ satisfying $dV=\eta|_C$. 
\end{definition}

\begin{remark}[Case of real polynomials]\label{vol:real}

Suppose that $P$ is real and irreducible over $\C$. We remark that in this case there is a preferred choice
for a volume function.

We define the complex conjugation on $(\C^*)^2$ by $\sigma(x,y)=(\overline{x},\overline{y})$. As the coefficients of $P$ are real, this conjugation preserves the curve $C$ and satisfies $\sigma^*\eta=-\eta$. Given a volume function $V_0:C\to \R$ such that $dV_0=\eta|_C$, the function $V=\frac{1}{2}(V_0-V_0\circ \sigma)$ is the unique volume function that is odd with respect to conjugation: $V \circ \sigma = -V$. 
\end{remark}

\begin{example}
We give some examples of exact polynomials:
\begin{enumerate}
  \item Any $A$-polynomial of a knot is exact. This can be proved by defining 
  directly the volume function which comes from volumes of representations 
  of cusped 3-manifolds. 
  There is an alternative proof using $K$-theoretic tools. 
  In any cases, we refer to \cite{CCGLS94}.
  \item $P(X,Y) = X+Y-1$. The volume function (in the sense of the previous remark) is $V(x,y) = -D(x)$ where $D$ is the Bloch-Wigner dilogarithm.
  \item If $\phi_5$ is the fifth cyclotomic polynomial, $P(X,Y) = Y-\phi_5(X)$ is exact. The volume function is 
  $V(x,y) = D(x) -\frac{1}{5} D(x^5)$.
  \item $P(X,Y) = 1 + X + Y + XY + X^2 + Y^2$.
  \item $P(X,Y) = 1 + iX + iY + XY$. 
\end{enumerate}
In fact, we will show later on (see section \ref{ss:Ktheory}) that each of these polynomials is a factor of the $A$-polynomial of some $3$-manifold; which $3$-manifold is unknown, apart from the last case (see \cite{Dun99} and Section \ref{ss:examples}).
\end{example}

Any volume function is clearly analytical on $C$ and extends to the completion of $C$ thanks to the following proposition.
\begin{proposition}
Let $P\in\C[X^{\pm 1},Y^{\pm 1}]$ be an exact polynomial and $V:C\to \R$ be a volume function. Then $V$ extends continuously to any projective model $\widehat{C}$ of $C$. 
\end{proposition}

\begin{proof}
Let $z$ be a point of $\widehat{C}\setminus C$. There exists a local coordinate $t$ around $z$ such that $x=t^p$ and $y=t^qF(t)$ where $p$ and $q$ are coprime integers and $F$ is a convergent series with $F(0)\ne 0$. We compute that in the coordinate $t=\rho e^{i\theta}$ one has:

\[\eta=p\log|F(\rho e^{i\theta})|d\theta-p \log(\rho)d \arg F(\rho e^{i\theta})\]
By integrating this form over a circle of radius $\rho$ and letting $\rho$ go to $0$, we find that the exactness of $\eta$ implies $\log|F(0)|=0$. This proves that we can factorise $\log(\rho)$ from the right hand side and conclude that $\eta$ is integrable at $0$, showing that its integral, $V$, extends continuously at $0$. 
\end{proof}

The previous proposition also gives the condition for $\eta$ to be exact in the neighbourhood of an ideal point. To give the formal interpretation, we recall the notion of tame symbol. Let $f,g$ be two meromorphic functions on a Riemann surface $X$ and $z$ be a point of $X$. Denoting by $v_z(f),v_z(g)$ the valuation of $f$ and $g$ at $z$ and by ev$_z$ the evaluation at $z$ we set 
\[\{f,g\}_z=(-1)^{v_z(f)v_z(g)}{\rm ev}_z\left(\frac{f^{v_z(g)}}{g^{v_z(f)}}\right)\in \C^*.\]

\begin{proposition}\label{tempere}
The form $\eta$ is exact in a neighborhood of $z\in \widehat{C}$ if and only if $|\{x,y\}_z|=1$. 
\end{proposition}
\begin{proof}
We compute the tame symbol $\{x,y\}_z$ using Puiseux coordinates: this gives $(-1)^{pq}F(0)^{-p}$. Hence, given a small circle $C_z$ around $z$ we get 
\[\int_{C_z}\eta=2\pi p \log |F(0)|=2\pi \log|\{x,y\}_z|. \]
\end{proof}

We give a name to polynomials verifying the condition in the proposition:
\begin{definition}
A polynomial $P\in\C[X^{\pm 1},Y^{\pm 1}]$ is said to be {\it tempered} if $|\{x,y\}_z|=1$ for all $z\in \widehat{C}$.
\end{definition}
We see from Proposition \ref{tempere} that an exact polynomial is tempered.
Let us leverage the proposition to describe better the set of tempered polynomials. Write $P=\sum_{(i,j)\in \Z^2} c_{i,j}X^iY^j$ and let $\Delta$ be the Newton polygon of $P$, that is the convex hull of the set of indices $(i,j)\in\Z^2$ such that $c_{i,j}\ne 0$.
The group $\SL_2(\Z)$ acts on $(\C^*)^2$ by monomial transformations and preserves the form $\eta$. It follows that the induced action on polynomials preserve the family of exact ones.

Given a polynomial $P$ with polygon $\Delta$, each side $s$ of the Newton polygon can be identified with the line $j=0$ with a monomial transformation. Collecting the monomials appearing along this line, we get a polynomial that we call the side polynomial $P_s\in\C[X^{\pm 1}]$. 

 \begin{proposition}
Let $P\in\C[X^{\pm 1},Y^{\pm 1}]$ be a polynomial with Newton polygon $\Delta$. The following assertions are equivalent.
\begin{enumerate}
\item $P$ is tempered.
\item For all sides $s$ of $\Delta$, the roots of the polynomial $P_s$ have modulus 1. 
\item The form $\eta|_C$ defines a cohomology class in $H^1(\widehat{C},\R)$.  
\end{enumerate}
\end{proposition}

\begin{proof}
The equivalence of $1.$ and $3.$ is clear from Proposition \ref{tempere}. 
 
Write $P=\sum_{i,j\in\Z}c_{ij}X^iY^j$ and let $z$ be an ideal point of $\hat{C}$. We consider a Newton-Puiseux coordinate as before, that is $x=t^p,y=t^qF(t)$. Then expanding $P(t^p,t^qF(t))$ into powers of $t$, we get a lower term of the form
\[\sum_{pi+qj=N}c_{i,j}F(0)^jt^{N}+o(t^N)\]
 where the line $pi+qj=N$ is a side of $\Delta$.  We get that $F(0)$ is a zero of the side polynomial $P_s=\sum_{pi+qj=N}c_{ij}X^j$. If $P$ is tempered, then $F(0)$ has modulus 1. Moreover, any root of any side polynomial gives rise to at least one Newton-Puiseux expansion and hence to one ideal point. This proves the equivalence of the first two items. 
\end{proof}

\begin{remark}[Case of real polynomials]
Suppose that $P$ is a real polynomial, tempered and irreducible over $\C$.

We then observe that $\eta$ satisfies $\sigma^*\eta=-\eta$ and hence its cohomology 
class belongs to the space $H^1(\widehat{C},\R)^-$ whose dimension is the genus of $\widehat{C}$. 
\end{remark}

\section{Extrema of the volume function}

Let $P$ be an exact polynomial and $C$ be the smooth part of the zero set of $P$ in $(\C^*)^2$. We denote by $\overline{C}$ the normalization of the zero set of $P$ in $(\C^*)^2$.  It is also the set of finite points of $\widehat{C}$, where neither $x$ nor $y$ have a zero or a pole. For this whole section, we choose a volume function on $C$.

We are interested in this section in the extrema of a volume function $V$ as it will turn out in the next
section that these extrema are the key input in a formula for the Mahler measure of exact polynomials. We will first describe two 
geometric tools to understand the variations of the volume, then go on with a study of critical points
for the volume before describing the extrema. At the end of this section, we are able to prove a first theorem on exact polynomial: they should have a zero in the torus $|x| = |y|=1$. 

\subsection{Amoeba and Gauss logarithmic map}

\begin{definition}
The {\it amoeba} of $C$ is the image of the map $\mu:\overline{C}\to \R^2$ defined by $\mu(x,y)=(\log|x|,\log|y|)$. 
\end{definition}

\begin{definition}
The {\it logarithmic Gauss map} is the map $\gamma:C\to \P^1(\C)$ defined by $\gamma(x,y)=[x\partial_xP,y\partial_yP]$. 
\end{definition}

We observe that $\gamma$ extends to $\widehat{C}$ as a holomorphic function. There is a relation between these two notions as shown in the following proposition, taken from \cite{mik00}.

\begin{proposition}\label{mik}
Let $C\subset(\C^*)^2$ be the smooth part of the curve defined by $P\in \C[X^{\pm 1},Y^{\pm 1}]$ and set $C_\R=C\cap \R^2$. Then 
\[C_\R\subset \{(x,y)\in C, d\mu \text{ is not invertible }\}=\gamma^{-1}(\P^1(\R))\]
\end{proposition}
\begin{proof}
Let $z=(x,y)$ be a point of $C$ and consider the zero set of the function $(u,v)\mapsto P(xe^u,ye^v)$ for $u$ and $v$ small. It defines a smooth subvariety around $0$ whose tangent space is given by $x\partial_x P(x,y) u+y\partial_y P(x,y) v=0$. In these coordinates, the derivative of $\mu$ is simply the map $(u,v)\mapsto (\re(u),\re(v))$. 
This map is non invertible if and only if there exists $(u,v)\in \R^2$ such that $x\partial_x P(x,y) u+y\partial_y P(x,y) v=0$. This is equivalent to $\gamma(x,y)$ being in $\P^1(\R)$. 
\end{proof}

We can even specify where $\mu$ preserves orientation, recalling that $C$ is naturally oriented, being a complex curve. Here comes a convention: 

\begin{definition}
A non-zero vector in $\C^2$ may be written $w=u+iv$ for $u,v\in\R^2$. We will say that $[w]\in \P^1_{\pm}(\C)$ if $\mp \det(u,v)>0$. 
In coordinates, $[z,1]\in \P^1_{\pm}(\C)$ if $\pm \im(z)>0$. 
\end{definition}

\begin{proposition}[Sequel of Proposition \ref{mik}]
In the same settings, for any $\epsilon\in\{\pm 1\}$ and $z\in P^1_{\epsilon}(\C)$, the differential $d_z\mu$ preserves the orientation if $\epsilon=1$ and reverses it if $\epsilon=-1$. 
\end{proposition}
\begin{proof}
Take a non-zero solution $(u,v)$ of the equation $$x\partial_x P(x,y) u+y\partial_y P(x,y) v=0.$$ Then an oriented basis of $T_zC$ is given by $(u,v), (iu,iv)$. The Jacobian of $\mu$ at $z$ in this basis is $\re(v)\im(u)-\re(u)\im(v)=\im(u\overline{v})$. This number has the same sign as $\im \frac{x\partial_xP}{y\partial_yP}$. 
\end{proof}

With these two concepts at hand, we proceed with the study of the critical points of the volume.

\subsection{Critical points of $V$}

We now look at  the volume function $V$ on $C$ and spot its critical points.

\begin{proposition} \label{critique}
A point $(x,y)\in C$ is a critical point of $V$ if and only if the following equation holds:
\[\log|x| x\partial_xP(x,y)+\log|y|y\partial_y P(x,y)=0.\]
In other terms, a point $(x,y)$ of $C$ is a zero of $\eta$ if either $\mu(x,y)=0$ or the two vectors of $\R^2$ $\mu(x,y)$ and $\gamma(x,y)$ are projectively dual. 
\end{proposition}

\begin{proof}
We consider as before a point of the form $(xe^u,ye^v)$ belonging to $C$. Then at first order we have $u x\partial_xP(x,y)+v y\partial_y P(x,y)=0$ whereas $\eta_{(x,y)}(u,v)=-\log|x|\im v +\log|y| \im u$. The conclusion follows.
\end{proof}

We also consider a projective model $\hat C$. We introduce a topological notion to describe the volume function at ideal or ramification points:

\begin{definition}
We will say that a continuous real function $f$ on a topological manifold $X$ has a {\it saddle} of order $k$ at $x$ if there is a local coordinate $z$ centered at $x$ such that $f(z)=f(x)+\re(z^k)+o(z^k)$. 
In particular $x$ is not an extremum of $f$ if $k>1$. 
\end{definition}

Notice that if $X$ and $f$ are smooth and $df(x)\ne 0$, then $f$ has a saddle of order $1$ at $x$. If $f$ has Morse critical point of index $1$ at $x$, then it has a saddle of order $2$. 

The next proposition shows that the behaviour of $V$ at an ideal point is quite simple.
\begin{proposition}\label{ideal}
Let $z$ be a point of $\widehat{C}\setminus \overline{C}$ and denote by $k$ the order of ramification of $\gamma$ at $z$. 
If $\gamma$ is not constant around $z$,  $V$ has a saddle point of index $k$ at $z$. 
\end{proposition}
\begin{proof}
Up to exchanging $x$ and $y$, we can choose Newton-Puiseux coordinates around $z\in \widehat{C}\setminus\overline{C}$ of the form $x=t^p,y=t^qF(t)$ with $p\ne 0$ and $|F(0)|=1$. 

Derivating the equality $P(t^p,t^qF(t))=0$ 
and writing $\gamma=\frac{x\partial_xP}{y\partial_yP}$ we get 
$$-\gamma(t)=\frac{q}{p}+\frac{tF'(t)}{pF(t)}.$$

We deduce from it that we can write $F(t)=e^{i\phi+a_k t^k+O(t^{k+1})}$ with $k>0$ and $a_k\ne 0$ unless $\gamma$ is constant around $t=0$. 
Plugging the formulas for $x$ and $y$ into the differential of $V$ we get:
\[\frac{1}{p}V'(t)=\log |F(t)| d\arg(t)-\log|t| d\arg F(t).\]

Integrating along the ray $[0,t]$ and then by parts we get the following expression which shows that $V$ has a saddle of order $k$: 
\[V(t)=V(z)-p\log|t|\im (a_kt^k)+O( t^{k+1}\log t)\]

\end{proof}

Note in particular that $V$ cannot have an extremum at an ideal point of $\hat C$, unless it is constant.
 
\subsection{Extrema of the volume}

We now leverage the study of critical points to see that the extrema of the volume only happen above the
torus $|x|=|y|=1$ in $\overline{C}$.

\begin{proposition}\label{extrema}
Let $z$ be a point of $\overline{C}\setminus C$ mapping to $(x_0,y_0)$. We denote by $k$ the order of ramification at $z$ of the map $(x,y):\overline{C}\to (\C^*)^2$. Denote by $\gamma:\overline{C}\to \P^1(\C)=\C\cup\{\infty\}$ the logarithmic Gauss map and suppose that $\gamma(z)\ne \infty$. We denote by $l$ the order of ramification of $\gamma$ at $z$. 
\begin{enumerate}
\item If $\log|x_0|\gamma(z)+\log|y_0|\ne 0$ then $V$ has a saddle of order $k$ at $z$. 
\item If $\log|x_0|\gamma(z)+\log|y_0|= 0$ and $\mu(z)\ne 0$ then $V$ has a saddle of order $k+l$ at $z$. 
\item If $\mu(z)=0$ and $\gamma(z)\notin \R$ then $V$ has a local maximum at $z$ if $\im \gamma(z)<0$ and a local minimum if $\im\gamma(z)>0$. 
\item If $\mu(z)=0$ and $\gamma(z) \in \R$ then $V$ does not have a local extremum at $z$.  
\end{enumerate}
\end{proposition}
\begin{proof}
Up to a monomial transformation, we can find a local coordinate around $z$ such that $x=x_0e^{t^k}$ and $y=y_0e^{t^kF(t)}$ where $F(t)=F_0+F_lt^l+O(t^{l+1})$. 
Again, derivating the equation $P(x,y)=0$ gives 
\[-\gamma(t)=F(t)+\frac{t}{k}F'(t)=F_0+F_l(1+l/k)t^l+O(t^{l+1}).\]
Plugging the formulas of $x$ and $y$ in the derivative of $V$ one get:
\begin{align*}V(t)=V(z)&+\im(\log|y_0|t^k-\log|x_0| t^k F(t))\\
&+\int_0^t \left(\re(t^kF(t))d\im t^k-\re t^k d\im(t^kF(t))\right).
\end{align*}
If $\log|y_0|-\log|x_0|F_0\ne 0$ then $V$ has a saddle of order $k$ as before. 

Suppose from now that this equation does hold and $\log|x_0|\ne 0$. This implies that $F_0$ is real. 
The next term in the first expression is $-\im(\log|x_0|F_lt^{k+l})$ whereas  the integral has the order of $t^{2k+l}$. We conclude in that case that $V$ has a saddle of order $k+l$. 

Suppose now that $\log|x_0|=\log|y_0|=0$. Then the first expression vanishes identically. 
If $F_0$ is not real, the first order in the integral is equal to $-\frac{1}{2}|t|^{2k}\im F_0$. In that case, $V$ has a maximum if $\im F_0>0$ and a minimum if $\im F_0<0$. 

Suppose now that $F_0$ is real so that this term vanishes. The next term is 
\[-\frac{|t|^{2k+l}r_l}{2k+l}\left(l\cos(k\theta)\sin(\phi_l+(k+l)\theta)+k\sin(\phi_l+k\theta)
\right)\]
where we have written $F_l=r_le^{i\phi_l}$. The trigonometric expression can be further expanded as $-\frac{l}{2}\sin(\phi_l+(2k+l)\theta)-\frac{l}{2}\sin(\phi_l+l\theta)-k\sin(\phi_l+k\theta)$. This expression cannot vanish and its integral over $\theta$ vanishes. This proves that $V$ does not have an extremum at $z_0$. 
\end{proof}

The fourth case will be explored and described more precisely in section 
\ref{ss:maxtancurves}. Moreover, from the previous discussion, we can describe easily the case where the volume is constant:
\begin{coro}
If $P$ is irreducible and $V$ is constant on $C$, then $\gamma$ is constant. As $\gamma$ take rational values at ideal points, this constant should be rational and one can write $P=X^pY^q-\lambda$ for some non-zero complex number $\lambda$. It is equivalent to say that $C$ is a translation of a sub-torus of $(\C^*)^2$ or that $\Delta(P)$ has non-empty interior. 
\end{coro}

There is a nice way to understand the properties of $V$ by looking at the amoeba of $C$. One can define gradient lines of $V$ as integral lines of the distribution $\star \eta$ where $\star$ denotes the Hodge star on $C$. A direct computation shows that 
\[\star \eta=\log|y|d\log|x|-\log|x|d\log|y|=\mu^* (vdu-udv)\]
where $u$ and $v$ are the coordinates in the image of $\mu$. This shows that gradient lines of $V$ project to half-lines in the amoeba of $C$. Moreover, flowing from the origin on the half-line, the volume is increasing if $\im \gamma>0$ or if $\mu$ preserves the orientation and is decreasing if $\im \gamma<0$ or if $\mu$ reverses the orientation.

Here is the example of the polynomial $P=X+Y-1$. The volume function is $V(x,y)=-D(x)$ where $D$ is the Bloch-Wigner Dilogarithm. Clearly, there are two points in $\mu^{-1}(0,0)$ which are $(e^{i\pi/3},e^{-i\pi/3})$ and its conjugate. They correspond the extrema of the function $D$, that is the maximal volume of a hyperbolic tetrahedron.

From Propositions \ref{ideal} and \ref{extrema}, we see that the extrema of $V$ can only occur in $\mu^{-1}(0,0)$. Hence we get the following theorem:
\begin{theorem}\label{intertore}
Let $P\in\C[X^{\pm 1},Y^{\pm 1}]$ be an exact irreducible polynomial such that the corresponding volume function is not constant. Then there exists $(x,y)\in (\C^*)^2$ such that $|x|=|y|=1$ and $P(x,y)=0$. 
\end{theorem}

We will build upon this theorem in Section \ref{finding}.

\section{Mahler measure of exact polynomials}\label{mahler}

We describe in this section the Mahler measure of exact polynomials, in a spirit similar to \cite{BoydRodriguezVillegas} but 
focusing on extrema and more broadly critical points of the volume function.

\subsection{Mahler measure and the volume function}

The Mahler measure of an exact polynomial is computed by integrating the form $\eta$ on a collection of
arcs inside $\hat C$ that we now define.

\begin{lemma}\label{demidroite}
Let $P\in \C[X^{\pm 1},Y^{\pm 1}]$ be an exact polynomial whose Newton polygon has non empty interior. Let $\overline{C}$ be the normalisation of the zero set of $P$ in $(\C^*)^2$ and consider the natural coordinate map $(x,y):\overline{C}\to(\C^*)^2$. 
Then up to a monomial transformation, the subset 
\[\Gamma=\{z\in \overline{C}, |x(z)|=1,|y(z)|>1\}\]
is a finite collection of arcs such that the volume function $V:\Gamma\to \R$ is monotonic on each interval.
Moreover, only a finite number of monomial transformations applied to $P$ do not satisfy these conditions. 
\end{lemma}
\begin{proof}
First, one may visualise $\Gamma$ as the preimage of the half-line $u=0,v>0$ by the map $\mu$. Applying a monomial transformation amounts in taking instead the preimage of any rational half-line. As a consequence, we can avoid any finite set of $C$. Hence we suppose that the half-line $u=0,v>0$ avoids all ideal points, singularities of $C$ and critical points of $V$ except those with $\mu=0$. 

Hence for all $z\in \Gamma$ with $\mu(z)\ne (0,0)$, the map $z\mapsto x(z)$ is not ramified. Indeed, if it were so, then we would have $\partial_y P(x,y)=0$ and $\gamma(z)=\infty$. As $\log|x|=0$, Proposition \ref{critique} implies that $z$ is a critical point of $V$, which we excluded. It follows that $\Gamma$ is smooth at these points. 
Let us show that $V|_\Gamma$ is not critical either: by definition we have $dV=\log|y|d\arg(x)$. As we restricted to $\log|y|\ne 0$ and $|x|=1$, this form does not vanish. 
\end{proof}

Let us compute the Mahler measure of $P$ given by 
\[m(P)=\frac{1}{(2\pi)^2}\int_0^{2\pi}\int_0^{2\pi}\log|P(e^{i\theta},e^{i\phi})|d\theta d\phi.\]
By Fubini theorem, setting $P_\theta(y)=P(e^{i\theta},y)$ we get 
\[m(P)=\frac{1}{2\pi}\int_0^{2\pi}m(P_\theta)d\theta.\]
Here $m(P_\theta)$ is the one-dimensional Mahler measure which can be computed thanks to Jensen formula:
\[m(P_\theta)=\log |a(\theta)|+\sum_{i}\log^+|y_i(\theta)|.\]
In this formula we wrote $P_\theta(y)=a(\theta)\prod_i(y-y_i(\theta))$ and as usual $\log^+|y|=\max(0,\log|y|)$.

As the Mahler measure of $P$ is obviously invariant by monomial transformations, we can suppose that $a$ is constant: this is equivalent to saying that the Newton polygon $\Delta$ does not have a top horizontal slope. 
We also observe that the pair $(e^{i\theta},y_i(\theta))$ belongs precisely to the subset $\Gamma$ of lemma \ref{demidroite}. Hence we can suppose that the integration goes along the completion $\overline{\Gamma}$ of $\Gamma$ that we orient in the direction of increasing $V$. 

We find 
\[m(P)=\log|a|+\frac{1}{2\pi}\int_{\overline{\Gamma}}\eta=\log|a|+\frac{1}{2\pi}V(\partial\overline{\Gamma}).\]

\begin{remark} We observe that given an exact polynomial $P$, all its coefficients in the corners of the Newton polygon are equal in absolute value. This is simply because the slope polynomial having roots of modulus one, their product also has modulus one. Hence, the extremal coefficients of these polynomial are equal in modulus. We will denote by $c(P)\in (0,+\infty)$ the absolute value of these corner coefficients. 
\end{remark}

\subsection{A formula for the Mahler measure -- the generic case}

In this section we give a first formula for the Mahler measure of an exact polynomial under some hypothesis
on the polynomial. Generically these hypothesis are fulfilled. Unfortunately, $A$-polynomials of hyperbolic cusped manifolds do not satisfy it in general due to extra symmetries. We will explain later on how to compute the Mahler measure in general. 
\begin{definition}
We will say that an exact polynomial $P\in\C[X^{\pm 1},Y^{\pm 1}]$ is {\it regular} if for every $z$ in $\overline{C}$ such that $\mu(z)=(0,0)$ we have $\gamma(z)\notin \P^1(\R)$. 
\end{definition}
Hence any such point $z$ is either a local maximum or minimum of the volume function $V$. We will denote by $\epsilon(z)$ minus the sign of the imaginary part of $\gamma(z)$ and by $k(z)$ the ramification at $z$ of the map $(x,y):\overline{C}\to(\C^*)^2$

\begin{theorem}\label{formule}
Let $P\in\C[X^{\pm 1},Y^{\pm 1}]$ be a regular exact polynomial. We have the formula:
\[m(P)=\log c(P)+\frac{1}{2\pi}\sum_{z\in \mu^{-1}(0,0)} k(z)\epsilon(z)V(z)\]
\end{theorem}
\begin{proof}
We analyse for each $z\in \mu^{-1}(0,0)$ how many branches of $\overline{\Gamma}$ end at $z$. Take logarithmic Newton-Puiseux coordinate: $x=e^{t^k},y=e^{t^kF(t)}$ with $\im F(0)\ne 0$: indeed $F(0)$ is the slope $\gamma(z)$. 

The manifold $\Gamma$ is defined by $|x(t)|=1$ and $|y(t)|>1$ hence writing $t=\rho e^{i\theta}$ we should have $\cos(k\theta)=0$, hence at most $2k$ branches of $\overline{\Gamma}$ end at $z$. 
We select only those branches for which $\ln|y(t)|>1$ that is $\re t^kF(t)=-\rho^k\sin(k\theta)\im F(0)+O(\rho^{k+1})$. In any case, we find $k$ branches of $\Gamma$ adjacent to $z$. 
\end{proof}

We get a lower bound on the Mahler measure in the \emph{irreducible} case. 
\begin{theorem}\label{inegalite}
Let $P\in \C[X^{\pm 1},Y^{\pm 1}]$ be a regular exact polynomial irreducible over $\C$ normalized so that $c(P)=1$.  
Then we have the inequality
\[2\pi m(P)\ge \max V-\min V.\]
\end{theorem}
\begin{proof}
Let $z_1,\ldots, z_n$ be the local maxima of $V$ ordered so that we have $V(z_1)\le \cdots\le V(z_n)$ and denote by $k_1,\ldots,k_n$ the corresponding ramifications orders. We also denote by $t_1,\ldots,t_m$ and $l_1,\ldots,l_m$ the data corresponding to the local minima of $V$. 
Lemma \ref{demidroite} provides a collection of arcs $\overline{\Gamma}$ as the completion of $\mu^{-1}(\{0\}\times (0,+\infty))$ and the proof shows that we could have taken any rational half-line except a finite number of them. Denote by $0<\theta_1<\cdots<\theta_m<2\pi$ the arguments of these forbidden half-lines. For any $\alpha\in \R/2\pi\Z$ distinct from these arguments, we denote by $\overline{\Gamma}_\alpha$ the completion of preimage of the half-line with argument $\alpha$. We always have 

\[2\pi m(P)= V(\partial \overline{\Gamma}_\alpha)=\sum_{i=1}^n k_i V(z_i)-\sum_{i=1}^m l_i V(t_i)\]

Set $I_\alpha=V(\overline{\Gamma}_\alpha)\subset [\min V,\max V]$. As $V$ is increasing along the components of $\overline{\Gamma}_\alpha$ we have $2\pi m(P)\ge \lambda(I_\alpha)$ where $\lambda$ denotes the Lebesgue measure. Suppose by contradiction that we have $2\pi m(P)<\max V-\min V$. Observing that all local extrema of $V$ belong to $I_\alpha$, we see that $I_\alpha$ has a "hole", meaning that there exists $x\in (\min V,\max V)$ such that $x\notin I_\alpha$. 

Hence $\Gamma_\alpha$ splits into two parts, mapping to $(\min V,x)$ and $(x,\max V)$ respectively. We should have as many maxima as minima above $x$: formally this implies that $\sum_{i, V(z_i)>x}k_i=\sum_{i, V(t_i)>x}l_i$. Taking any other collection $\overline{\Gamma}_\beta$: the number $\sum_{i, V(t_i)>x}l_i$ also corresponds to the number of increasing branches of $\overline{\Gamma}_\beta$ starting from a point above $x$. As there are as many arriving points above $x$, this implies that no other branch of $\overline{\Gamma}_\beta$ can come from below $x$. We conclude that $\forall\beta\notin\{\theta_1,\ldots,\theta_m\}, x\notin I_\beta$.

As the extremal points of $I_\beta$ are one of the local extrema of $V$, there exists $\epsilon>0$ such that $(x-\epsilon,x+\epsilon)\cap I_\beta=\emptyset$. However it is clear that $\bigcup_{\beta\notin\{\theta_1,\ldots,\theta_m\}} \overline{\Gamma}_\beta$ is dense in $\hat{C}$ hence we conclude that 
\[V(\hat{C})\cap (x-\epsilon,x+\epsilon)=\emptyset.\]
As $P$ is irreducible, $\hat{C}$ is connected and this contradicts the continuity of $V$.  
\end{proof}
Remark that if we have the equality $2\pi m(P)=\max V-\min V$, this means that the map $V:\overline{\Gamma}_\alpha\to [\min V,\max V]$ does not have overlaps except at the branching points. This is possible only if $k_i=l_i=1$ for all $i$ and $V(t_1)<V(z_1)=V(t_2)<\cdots V(z_{n-1})=V(t_n)<V(z_n)$. 

\subsection{A formula for the Mahler measure -- the general case}

We now extend our previous formula to the more complicated case of a general exact polynomial. We shall to define an index of a critical point of the volume on the 
torus which will play the role of the numbers $k_i \epsilon_i$ in the previous version. The general formula will involve the value of the volume function on critical points that project inside the torus $|x| = |y| = 1$.

\subsubsection{Maximally tangent curves}\label{ss:maxtancurves}

Let $z \in \hat C$ be a point such that $\mu(z) = (0,0)$. Let $k$ be the ramification order at $z$ of the map $\hat C \to (\C^*)^2$. 

In the case $\im(\gamma(z))\neq 0$, we call \emph{index of $z$} the number $k(z)\epsilon(z)$ which appears in the formula in the previous subsection.
From now on, we suppose that $\gamma(z)\in \R\setminus\{0\}$.

Let $\xi$ be an element of $P(T_z \hat C)$ - i.e. a real line in $T_z \hat C$. We define below its \emph{order of tangency to the torus} as the maximal order of tangency with the torus of a curve in $\hat C$ tangent to $\xi$. Recall that the points in $\hat C$ which project to the torus are exactly those satisfying $\mu(z)=(0,0)$.
\begin{definition}
For $\xi \in P(T_z \hat C)$, the \emph{order} of $\xi$ is the number $\ord(\xi) \in \mathbb N \cup{+\infty}$ defined by:
$$\ord(\xi) = \sup\left\{  l \in \mathbb N \left|\begin{matrix}\exists\alpha : ]-\epsilon,\epsilon[\to \hat C \textrm{ smooth curve satisfying  }\\
                                      \alpha_0 = z,\alpha'_0 \in \xi\setminus\{0\} \textrm{ and }\mu(\alpha_s)=O(s^l)\end{matrix}\right.\right\}.$$
                                      
Moreover a smooth curve with maximal $l$ above will be said \emph{maximaly tangent to the torus} in the direction $\xi$.
\end{definition}

The following proposition shows that there is a relation between $k(z)$ and the various orders of tangency at $z$. 

\begin{proposition}
Let $z\in \hat C$ satisfy $\mu(z)=(0,0)$ and set $k=k(z)$. For every direction $\xi$ but exactly $k$, we have $\ord(\xi) = k$. For the $k$ remaining directions $\xi_1, \ldots, \xi_k$, the order is $\geq k+1$.
\end{proposition}

\begin{proof}
Take a local coordinate $t$ such as $(x = x_0 e^{it^k}, y = y_0 e^{i t^kF(t)})$. Consider a smooth curve $\alpha:(-\epsilon,\epsilon)\to \hat{C}$ with $\alpha_0 = z$. In the coordinate $t$ we have $\alpha_0=0$ and $\alpha'_0=\xi\ne 0$. Recall that as $\gamma(z)$ is real, we have $F(0)\in \mathbb R$.
It yields: 
$$\mu(\alpha_s) = (\ln|x(\alpha_s)|,\ln|y(\alpha_s))|) = -s^k(\im(\xi^k), F(0) \im(\xi^k)) + O(s^{k+1}).$$ 
We see that $\mu(\alpha_s)=O(s^{k+1})$ iff $\im(\xi^k) = 0$, defining the $k$ particular directions. In the coordinate $t$, these directions are the $k$ lines of angle $\theta_j$ with $k \theta_j \equiv 0 [\pi]$.
\end{proof}

\subsubsection{Variation of the volume}

We show here that the variation of the volume on maximally tangent curves in
the directions $\xi_j$ only depends on the direction. We shall use again a local coordinate 
$t$ such that $(x= x_0e^{it^k}, y=y_0 e^{it^k F(t)})$, where $F(t) = \sum\limits_{i\ge 0} F_it^i$. Fix the direction $\xi_j$, corresponding in this coordinate to the angle $\theta_j$. Let 
$\alpha_s$ be a tangent curve in the direction $\xi_j$, written in coordinates in the form:
$$\alpha_s = a_0s(1 + \sum_{i\geq 1} a_i s^i).$$  Up to rescaling the real paramater $s$, we may assume $a_0 = e^{i\theta_j}$.

Let us begin by a characterisation in the chosen coordinate of maximally tangent curves:

\begin{lemma}
The order of tangency of $\xi_j$ is $\ord(\xi_j)=k+l$ where $l$ is the first integer such that $F_le^{(l+k)i\theta_j}$ is not real.

Moreover the curve $\alpha_s = e^{i\theta_j}s(1 + a_1 s +\ldots)$ is 
maximally tangent to the torus iff the coefficients $a_1, \ldots, 
a_{\ord(\xi_j) - k}$ are real.
\end{lemma}

\begin{proof}
Let $l$ be the first integer such that $F_{l}e^{(k+l)i\theta_j}$ is not real. We have to prove that $k+l$ is the order of $\xi_j$.

First consider the curve $\alpha_s=s e^{i\theta_j}$. Applying the map $\mu$, we compute  
$\mu(\alpha_s) = -s^{k+l}(0,\im(F_le^{(k+l)i\theta_j})) + o(s^{k+l})$. Hence $\ord(\xi_j) \geq k+l$.

Second, consider a curve $\alpha_s$ written as above $\alpha_s = e^{i\theta_j}s(1 + a_1 s +\ldots)$ whose order of tangency to the torus is $\geq k+l$. Let $r$ be the smallest (if it exists) integer such that $a_r$ is non real. The first non-vanishing term (if it exists) of $\ln\left|x(\alpha_s)\right|$ is $- ks^{r+k}\im(a_r)+\ldots$. By the assumtion on the order of tangency, we get $r \geq l$.

If we have $r = l$, we may compute the term of order $k+l$ in
$\mu(\alpha_s)$: it is equal to $s^{k+l} (-k\im(a_r), -\im(ka_r+F_le^{(l+k)i\theta_j}))$. 
One the other hand, if we have $r>l$, this term equals $s^{k+l} (0, -\im(F_le^{(l+k)i\theta_j}))$. In any case, $\mu(\alpha_s)$ has order $k+l$. This proves the lemma. 
\end{proof}

In other terms, up to a \emph{real} reparametrisation of the variable $s$, we may write 
any maximally tangent curve in the direction $\xi_j$ in the form:
$\alpha_s =  se^{i\theta_k}(1 + a_ls^l + \ldots)$. Note that any curve verifying $|x|=1$ along the curve is maximally tangent: for a curve $\alpha_s = e^{i\theta}s(1 + a_1 s +\ldots)$ to verify $|x(\alpha_s)| = 1$, we must have $\theta = \theta_j$ for some $j$ and every coefficient $a_m$ real.

\begin{remark}\label{rem:maxtang}
From the above discussion, any arc $\gamma$ in $\mu^{-1}(\{ 0 \}\times \R)$ is maximally tangent to the torus at any point of $\gamma$ in the torus. Moreover, at a point $z$ in the torus,
$\mu^{-1}(\{ 0 \}\times \R)$ contains $k$ maximally tangent curves, one for each direction.

As the definition of maximally tangent curves is invariant under monomial
transformation, we see 
that this is true for any $\mu^{-1}(d)$ where $d$ is a line of rational 
slope through $0$ in $\R^2$. 
\end{remark}

The behaviour of the volume function along a maximally tangent curve $\alpha$ only depends on the order $\ord(\xi_j)$ and the sign $\epsilon_j$ of $\im(F_{\ord(\xi_j)-k}e^{\ord(\xi_j)i\theta_j})$, called the sign of the direction. At this point, it seems to depend on the local coordinate chosen, but the following lemma shows the converse:

\begin{proposition}\label{prop:varvolmaxtang}
Let $\alpha$ be a maximal tangent curve in the direction $\xi_j$ of order $\ord(\xi_j)$ and sign $\epsilon_j$. Then:
\begin{enumerate}
  \item if $k+\ord(\xi_j)= +\infty$, then $s\mapsto V(\alpha_s)$ is constant.
  \item if $k+\ord(\xi_j)$ is even, and $\epsilon_j >0$, then $s\mapsto V(\alpha_s)$ has a strict local maximum at $0$.
  \item if $k+\ord(\xi_j)$ is even, and $\epsilon_j <0$, then $s\mapsto V(\alpha_s)$ has a strict local minimum at $0$.
  \item if $k+\ord(\xi_j)$ is odd, $s\mapsto V(\alpha_s)$ is strictly monotonous.
\end{enumerate}
\end{proposition}

\begin{proof}
Let $l = \ord(\xi_j)-k$ and $\theta = \theta_j$.
The statement does not depend on a real reparametrisation of the variable $s$, so we assume that we can write in coordinate $\alpha_s =  se^{i\theta_k}(1 + \alpha_ls^l + \ldots)$.

Using the formula for $dV$, this parametrisation and computations similar to those of the previous lemma, we get that the first non vanishing term in 
$\frac{d}{ds}V(\alpha_s)$ is of order $2k+l-1$. Indeed, 
this first non vanishing term is obtained by looking at the terms 
of order $k+l$ of $\ln|x|$ and $\ln|y|$ and $k-1$ of $d\arg(x)$ and $d\arg(y)$.
There may be some vanishing terms in the following expressions but we may anyway write (recall that $e^{ik\theta}$ and $F_0$ are real):
\begin{enumerate}
\item $\ln|x| = s^{k+l} (-\im(e^{ik\theta} \alpha_l))+\ldots$
\item $d\arg(x) = ke^{ik\theta} s^{k-1}+\ldots$
\item $\ln|y| = s^{k+l} (-\im(e^{i(k+l)\theta} F_l)- \im(F_0 ke^{ik\theta}\alpha_l))+\ldots$
\item $d\arg(y) = ke^{ik\theta} F_0 s^{k-1}+\ldots$
\end{enumerate}

We deduce:
\begin{eqnarray*}
\frac{d}{ds}V(\alpha_s) = &\ln|y|d\arg(x) - \ln|x|d\arg(y) \\
=& ks^{2k+l-1}\left(-\im(F_{l} e^{(l+k)i\theta_j})\right)+ \ldots
\end{eqnarray*} 

So the sign of this derivative is $-\epsilon$ times the sign of $s^{2k+l-1}$. This proves the proposition.
\end{proof}

In the cases (2), (3), (4), we call the direction $\xi_j$ respectively maximizing, 
minimizing, or monotonous.
We are ready to define the index of the point $z \in \hat C$, which extends the regular case:
\begin{definition}
Let $z$ be a point in $\hat C$ with $\mu(z)=(0,0)$.

If $\gamma(z)$ is not real, then let $k$ be the ramification at $z$ of $z\to (x,y)$ and $\epsilon$ be minus the sign of $\im(\gamma(z))$. Then the index of $z$, denoted by $\textrm{Ind}(z)$, is $k\epsilon$.

If $\gamma(z)$ is real, the index of $z$, denoted by $\textrm{Ind}(z)$, is the number of maximizing directions minus the number of minimizing directions.
\end{definition}

\begin{remark}\label{rem:maxtan2}
As discussed in the previous remark \ref{rem:maxtang}, the index of a point $z$ is invariant under monomial transformations: this transformation sends minimizing (resp. maximizing) curves to minimizing (resp. maximizing) curves.
\end{remark}

Note that only a finite number of points $z\in \mu^{-1}(0,0)$ have a non trivial index. In particular, any 
non-singular point of a curve in the intersection of $C$ and the torus has index $0$: at such a point, 
there is only $1$ maximally tangent curve to the torus, which is included in the torus. Along this curve, 
the volume is constant.

\subsubsection{Mahler measure of exact polynomial, general case}

The index we defined above also describes the contributions of points in the torus to the Mahler measure as shown by the following theorem. 

\begin{theorem}\label{formule-generale}
Let $P\in\C[X^{\pm 1},Y^{\pm 1}]$ be an exact polynomial. We have the formula:
\[ m(P)=\log c(P)+\frac{1}{2\pi}\sum_{z\in \mu^{-1}(0,0)} \mathrm{Ind}(z)V(z)\]
\end{theorem}

\begin{proof}
First apply Lemma \ref{demidroite} so that $\Gamma=\mu^{-1}(\{0\}\times(0,+\infty))$ is a smooth collection of arcs with monotonic volume on each component. Up to applying yet another monomial transformation, we may assume that the slope $\gamma$ is never $0$ for an isolated or singular point $z$ in $\mu^{-1}(0,0)$. Let $k$ be the ramification order at $z$ of $z \mapsto (x,y)$. If $\gamma(z)$ is not real, then its contribution has already been understood in the 
proof of Theorem \ref{formule}: it is $k\epsilon V(z) = \textrm{Ind}(z) V(z)$.

We now assume that $\gamma(z)$ is not real. Let $D_{\rm real}(z)$ be the number of directions included in the torus, $D_{\min}(z)$ the number of minimizing directions, $D_{\max}(z)$ the number of maximizing directions, and $D_{\rm mon}(z)$ the number of monotonous directions. Note that
the sum of this four numbers is $k$ and that if the slope $\gamma(z)$ is not real, we have $D_{\min}(z) = k$ or $D_{\max}(z)=k$ depending on the sign of its imaginary part. The invariance of these four integers by monomial transformations (see Remark \ref{rem:maxtan2}) implies the following lemma:

\begin{lemma}
The arcs of $\Gamma$ containing $z$ are composed of half-curves maximally tangent to the torus, namely:
\begin{itemize}
\item[-] $D_{\min}(z)$ arcs are leaving from $z$, corresponding to the side in which the volume increases.
\item[-] $D_{\max}(z)$ arcs are arriving at $z$, corresponding to the side in which the volume increases.
\item[-] for each monotonous direction, there can be $0$ or $2$ arcs. In the second case, those are
both the half-curves, one leaving from $z$, the other arriving.
\end{itemize}
\end{lemma}

\begin{proof}
The point $z$ is in the torus. As noted in remark
\ref{rem:maxtang}, around $z$, the arcs of $\Gamma$ are included in the set $\mu^{-1}(\{0\}\times \R)$ which consists in $k$ maximally tangent curves at $z$.
But the variation of the volume along these $k$ curves has been described in proposition \ref{prop:varvolmaxtang}. This description gives the lemma: note that the maximally tangent curves completely
contained in the torus do not appear, as the volume is constant along such a curve and $\Gamma$ does not see them. 
\end{proof}

The monotonous arcs, if they appear, cancel immediately: one arc is leaving whenever the other arrives. So, the contribution of $z$ to the Mahler measure is $(D_{\max}(z)-D_{\min}(z)) V(z) = \textrm{Ind}(z)V(z)$. Summing over all $z$, we get the formula of the theorem:
$$m(P) = \log(c(P)) + \frac{1}{2\pi}\sum_{z\in \mu^{-1}(0,0)} \textrm{Ind}(z) V(z).$$
\end{proof}

Arguing as in the regular case, we get:
\begin{theorem}\label{inegalite-generale}
Let $P\in \C[X^{\pm 1},Y^{\pm 1}]$ be an exact polynomial irreducible over $\C$ normalized so that $c(P)=1$.  
Then we have the inequality
\[2\pi m(P)\ge \max V-\min V.\]
\end{theorem}

\section{Finding exact polynomials}\label{finding}

Choose a polygon $\Delta\in \R^2$ with integral corners and coefficients $c_{i,j}\in \C$ for $(i,j)\in\partial \Delta\cap \Z^2$ such that the side polynomials constructed from these coefficients have all their roots of modulus 1. Then any choice of coefficients $c=(c_{i,j})$ for $(i,j)\in \inter{\Delta}\cap \Z^2$ gives rise to a tempered polynomial $P_c$. The question we address in this section is: for which coefficients is $P_c$ exact?

It is well-known that generically, the curve $\widehat{C}_c$ associated to $P_c$ is smooth and has genus $N=\card (\inter{\Delta}\cap \Z^2)$. Hence, the cohomology class of $\eta|_{\widehat{C}_c}$ belongs to the $N$-dimensional space $H^1(\widehat{C}_c)$. Moreover, the period map being analytical, we expect that the $N$-th dimensional family $(P_c)$ is exact only for a finite number of $c$. 

This is trivial for polygons which have no interior point. In this section, we prove it for tempered families of elliptic curves, that is for polygons with one integral point. 

\begin{theorem}
Let $\Delta$ be a Newton polygon with $(0,0)$ as the unique interior point and let $(c_{i,j})_{(i,j)\in\partial \Delta\cap \Z^2}$ be a system of coefficients such that all side polynomials have roots which are simple and of modulus 1. Then for all but a finite number of values of $c\in \C$ the polynomial 
$$P_c=\sum_{(i,j)\in\partial \Delta\cap\Z^2}c_{i,j}X^iY^j -c$$ 
is exact. 
\end{theorem}
\begin{proof}
By Theorem \ref{intertore}, if $P_c$ is exact, there exists $x,y\in \C$ with $|x|=|y|=1$ and $P(x,y)=0$. As the coefficients $c_{i,j}$ are fixed, this implies that $c$ lives in a compact set. 
Let $c$ be a regular value of the map $P:(\C^*)^2\to \C$ defined by $P=\sum_{(i,j)\in\partial \Delta\cap\Z^2}c_{i,j}X^iY^j$ and $C_c$ be its zero set $\{(x,y) \in (\C^*)^2 | P(x,y)=c\}$. 

Let $[\eta_c]\in H^1(\widehat{C}_c,\R)$ be the cohomology class of the restriction of $\eta$ to $\widehat{C}_c$. There is a neighborhood $U$ of $c$ such that $P$ corestricted to $U$ is a trivial fibration. This allows to define the period map $\mathcal{P}:U\to H^1(\widehat{C}_c,\R)$. If we show that the differential of the period map is invertible at $c$, we will conclude that the set of $c's$ for which $P_c$ is exact is discrete and hence conclude the proof of the theorem.

We define $\omega=\frac{dx}{x}\wedge\frac{dy}{y}$ and $X_P$ the Hamiltonian vector field satisfying $i_{X_P}\omega=dP$. In formulas, $X_P=xy(\partial_yP\partial_x-\partial_xP\partial_y)$. 
This vector field is tangent to $C_c$ and does not vanish, we denote by $\alpha_P$ the holomorphic form satisfying $\alpha_P(X_P)=1$. By Lemma \ref{regular} below, this form is indeed holomorphic on the smooth projective model $\widehat{C}_c$ which has genus 1 and hence does not vanish. Hence the flow of $X_P$ gives the uniformization of $\widehat{C}_c$. 

Consider a smooth family of cycles $\gamma_t:S^1\to C_{c(t)}$ where $c(0)=c$ and $c'(0)=\lambda$. As $d\eta=-\im \omega$ we have by Stokes formula $\int_{\gamma_1}\eta-\int_{\gamma_0}\eta=-\int_{S^1\times[0,1]}\im(\gamma^*\omega)$. Letting $t$ going to $0$ we find that 

\[\mathcal{P}'(c)(\gamma_0)=\frac{d}{dt}\Big|_{t=0}\int_{\gamma_t}\eta=\int_{\gamma_0}\im (\lambda i_\xi\omega)\]
where $\xi$ is a vector field defined on $C_c$ such that $dP(\xi)=1$. 

If we take as $\gamma_0$ a periodic orbit of the flow of $X_P$ of complex period $T$, we will find $\int_{\gamma_0} i_\xi\omega=\int \omega(\xi,X_P)dt=\int dP(\xi)dt=T$. 
This proves that $\mathcal{P}'(c)\ne 0$ and hence the theorem. 
\end{proof}

\begin{lemma}\label{regular}
Let $P\in\C[X^{\pm 1},Y^{\pm 1}]$ be a polynomial with Newton polygon having one interior point and whose side polynomials have only simple roots of modulus 1. Suppose that the vanishing locus $C$ of $P$ on $(\C^*)^2$ is smooth and denote by $\widehat{C}$ its projective model.  
Then the form $\alpha_P$ dual to the Hamiltonian vector field $X_P$ of $P$ relative to the symplectic form $\omega=\frac{dx}{x}\wedge \frac{dy}{y}$ has no pole nor zero on $\widehat{C}$. 
\end{lemma}
\begin{proof}
As $dP$ does not vanish on $C$, the vector field $X_P$ neither. Hence, it is sufficient to show that $\alpha_P$ is holomorphic at each ideal point. 
Up to translations and monomial transformations, one can suppose that the interior point is the origin and the side of $\Delta$ we are looking at is $i=1$. Hence, we have Newton-Puiseux coordinates $x=t^{-1}, y=F(t)$ where $\sum_{i=1,j}c_{i,j}F(0)^j=0$. 
One can write $\alpha_P=\frac{dx}{xy\partial_yP}$. As $y\partial_yP=t^{-1}\sum_{i=1,j}c_{i,j}jF(0)^j+o(t)$ we get $y\partial_yP\sim\alpha t^{-1}$ and $\alpha_P\sim -dt\alpha^{-1}$. This proves that $\alpha_P$ is regular at the corresponding ideal point. 
 \end{proof}

\section{A-polynomials}

We focus in this section on A-polynomials as a particular class of exact polynomials. 
We gather different observations and examples. We first explain that
there is an algebraic criterion for being a $A$-polynomial which is effective in genus 0. Then, we come back to the
examples given in the first part, proving they are $A$-polynomials and computing their
Mahler measure. We add a few examples of $A$-polynomial of cusped hyperbolic $3$-manifold
to exhibit different possible behaviours. We then move on to give an interpretation of
the Mahler measure of some knot exteriors $M$ in terms of "lengths" of the filling geodesic in long Dehn surgeries of $M$.

From now on, $P$ will be a polynomial over $\bar \Q$ and often over $\Z$.

\subsection{A $K$-theoretic criterion for being an $A$-polynomial}\label{ss:Ktheory}

\begin{definition}
We say that an irreducible polynomial $P\in \overline{\Q}[X^{\pm},Y^{\pm}]$ is a $A$-factor if there exists a 3-manifold $M$ with toric boundary such that $P$ is a factor of $A_M$. 
\end{definition}
In the sequel, we will write $F_P={\rm Frac }\,\overline{\Q}[X^{\pm},Y^{\pm}]/(P)$. 

\begin{proposition}
The polynomial $P\in \overline{\Q}[X^{\pm},Y^{\pm}]$ is an $A$-factor if and only if the Steinberg symbol $\{X,Y\}$ vanishes in $K_2(F_P)\otimes \Q$
\end{proposition}
\begin{proof}
Suppose that $P$ is a factor of the $A$-polynomial of $M$. This implies that there exists a curve $C$ in the representation variety $\homo(\pi_1(M),\SL_2(\overline{\Q}))$ whose restriction is dense in the zero set of $P$. Consider $E$ the function field of $C$: it is a finite extension of $F_P$ and there is a tautological representation $\rho:\pi_1(M)\to \SL_2(E)$. Following the argument of \cite{CCGLS94}, p.59 it follows that $2\{X,Y\}=0$ in $K_2(E)$. As the map $K_2(F_P)\to K_2(E)$ is injective modulo torsion thanks to the transfer map, we conclude that $\{x,y\}=0\in K_2(F_P)\otimes \Q$. 

Reciprocally, suppose that $\{X,Y\}=0\in K_2(F_P)\otimes\Q$. Then, $K_2(\overline{F_P})=\lim\limits_{\to}K_2(E)$ where $E$ ranges over the finite extension of $F_P$ and it is known (Bass-Tate theorem) that this group is divisible, hence $2\{X,Y\}=0\in K_2(E)$ for some finite extension $E$ of $F_P$. 
Let $\rho:\pi_1(S^1\times S^1)\to \SL_2(E)$ be the representation that send $l$ to the diagonal matrix with entries $X,X^{-1}$ and $m$ to the diagonal matrix with entries $Y,Y^{-1}$. Then, $\rho_*([S^1\times S^1])=2\{X,Y\}=0\in H_2(\SL(E))=K_2(E)$. By Sah stability theorem (see \cite{W13}) this implies that $\rho_*([S^1\times S^1])=0\in H_2(\SL_2(E))$. By cobordism theory applied to the classifying space BSL$_2(E)$, this shows that there exists a manifold $M$ bounding $S^1\times S^1$ and a representation $\tilde{\rho}:\pi_1(M)\to \SL_2(E)$ extending $\rho$. By construction, $P$ will be a factor of $A_M$. 
\end{proof}

Hence, it is easy to recognise $A$-factors of genus 0 because by the localization formula for $K_2$ (see \cite{W13}, p.257), the criterion for being an $A$-factor reduces to the condition that all tame symbols $\{X,Y\}_z$ are torsion. Thus we get the following corollary:
\begin{coro}
An irreducible polynomial $P\in  \overline{\Q}[X^{\pm},Y^{\pm}]$ of genus 0 is an $A$-factor if and only if the roots of its side polynomials are roots of unity.
\end{coro}
It follows that the polynomials $X+Y-1,1+X+Y+XY+X^2+Y^2$ are $A$-factors but we don't know to which 3-manifolds they correspond. Indeed, the simplest non-trivial genus 0 $A$-factor we know is $1+iX+iY+XY$ which corresponds to the suspension of a punctured torus over the circle with monodromy $\begin{pmatrix}-1 & -2 \\ -2 & -5\end{pmatrix}$, see \cite{Dun99}.

\subsection{Examples}\label{ss:examples}

We review here some examples of exact polynomials given in the first section
and prove that they are indeed $A$-polynomials. We then proceed with the computations of their Mahler measure.

\subsubsection{$P_1 = X+Y-1$}

This example is directly related to the original one of Smyth \cite{Smyth}. On the curve defined by $P_1$, we have $\{x,y\} = \{x,1-x\} = 0$ in $K_2(F_{P_1})$ and
the previous proposition shows $P_1$ is an $A$-factor.

It is easy to see that the volume function on the curve is given by $-D(x)$. Indeed, we recognize its differential in the following expression: 
$$\eta(x,1-x)= \log|1-x|d\arg(x) - \log|x|d\arg(1-x).$$
Moreover, the only point of the curve on the torus are given by $x = e^{\pm \frac{i\pi}{3}}$.
They are the maximum and minimum points of the volume, with the value being plus or minus the volume $v_3$ of the regular ideal hyperbolic ideal. This whole 
discussion is done in \cite{BoydRodriguezVillegas} and builds upon a computation
by Smyth \cite{Smyth}. In these references, a further number theoretic description of the Mahler measure is also given. 

It is easy to see that the volume above the circle $|x|=1$ has no critical points: the Mahler measure is then computed using section \ref{mahler}:
$$m(P_1) = \frac{v_3}{\pi}.$$

\subsubsection{$P_2(X,Y) = Y-\phi_5(X)$}

This is an instance of examples treated by Boyd and Rodriguez-Villegas \cite{BRV-I}. On the curve $C_2$ defined by $P_2$, we have $\{x,y\} = \frac{1}{5}\{x^5,1-x^5\} - \{x,1-x\} = 0$ in $K_2(F_{P_1})$ and
the previous proposition shows $P_2$ is a (factor of) a $A$-polynomial.

As before, we see that the function $x \to V_2(x) = D(x)-\frac{1}{5} D(x^5)$ is the volume function. We  plot this function above the circle $|x|=1$ in Figure \ref{fig:P2}.
\begin{figure}[ht]
\begin{center}
\includegraphics[width = 8cm]{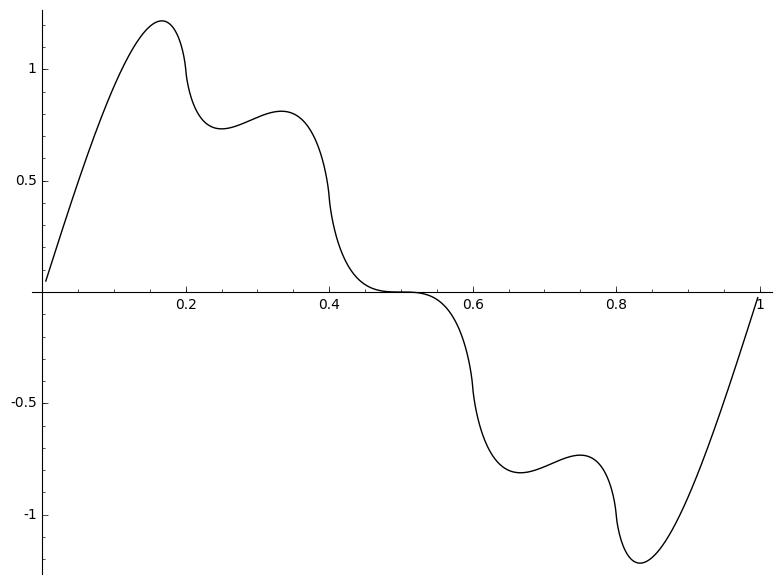}
\end{center}
\caption{Volume function for $P_2(x,y)$ above the circle $|x|=1$}\label{fig:P2}
\end{figure}

The intersection of $C_2$ with the torus $|x|=|y|=1$ consists of $7$ points, corresponding to $x=- 1, \pm i, \omega = \frac{1+i\sqrt{3}}{2},\omega^2,\omega^4,\omega^5$. This points are easily spotted on the figure as they correspond to critical points of the volume function. From the figure, it is quite clear that 
$$2\pi m(P_2) = V_2(\omega) - V_2(i) + V_2(\omega^2) - V_2(\omega^4) +V_2(i) -V_2(\omega^5).$$
Let us prove this from the Formula \ref{formule-generale}. The logarithmic Gauss map at these points may be computed: 
$$X\partial_X P_2 = -(X+2X^2+3X^3+4X^4)\textrm{ and } 
Y\partial_YP_2 = Y=1+X+X^2+X^3+X^4.$$
We then decide if the point is in $\P^1_+(\C)$, $\P^1(\R)$ or $\P^1_-(\C)$. Table \ref{table1} display the computation of this sign for $\omega$, $i$, $\omega^2$, $-1$. The three remaining points $\omega^4$, $-i$, $\omega^5$ are the complex conjugates of $\omega$, $i$, $\omega^2$, so the table is easily filled.

\begin{table}
\begin{center}
\begin{tabular}{c|c|c|c|c|}
$x$ & $\omega$ &  $i$ & $\omega^2$ & $-1$ \\
\hline
$\gamma(x,\phi_5(x))$ & $[-\frac 52 - 3i\sqrt{3},1]$ & $[-2+2i,1]$ & $[-2-i\sqrt{3},1]$ & $[-2,1]$ \\
\hline
sign & $\P^1_-(\C)$ & $\P^1_+(\C)$ & $\P^1_-(\C)$ & $\P^1(\R)$ 
\end{tabular}
\end{center}
\caption{Slope of points in the torus}\label{table1}
\end{table}
The sign is in accordance with what is shown on the figure: at a point in 
$\P^1_+(\C)$, the volume exhibits a local minimum. The point $-1$ of real slope is 
easily discarded in the formula: whatever its index is, its volume vanishes. So it does not contribute to the Mahler measure. Its index is indeed $0$: at this point, there is only one maximally tangent curve, which is monotonous.

Using the explicit formula for $V_2$, we get:
$$m(P_2) = \frac{2}{5\pi} \left(3D(\omega) +3D(\omega^2) - 2D(i)\right)$$

\subsubsection{$P_3(X,Y) = 1 + X + Y + XY + X^2 + Y^2$}
In this example, the curve $C_3$ has no real points. However we know that it must have points on the torus, 
as $P_3$ is exact. Indeed we find 8 such points: $(-1,i)$, $(-1,-i)$, $(i,-1)$, $(-i,-1)$, $(i,-i)$, $(-i,i)$,
$(\omega^2,\omega^4)$, $(\omega^4,\omega^2)$. All these points have non real slope and no ramification hence 
we get using symmetries and denoting by $V_3$ the volume function -- which is well-defined, see remark 
\ref{vol:real}:
$$ 2\pi m(P_3)=4V_3(-1,-i)+2V_3(i,-i)+2V_3(\omega^4,\omega^2).$$
This time, we do not provide an explicit computation using the Bloch-Wigner dilogarithm although such an expression should exist. 

\subsubsection{$P_4(X,Y) = 1 + iX + iY + XY$}

This example is (almost) already computed in \cite[Example 9]{BoydRodriguezVillegas}.
Here, using the same techniques as described above, we get that $\{x,y\} =\{-ix,1+ix\}-\{ix,1-ix\}$ vanishes modulo torsion on the curve defined by $P_4$. Moreover, it follows that the volume function is $D(ix)-D(-ix)$ for any point $(x,y)$ in the curve.

The intersection between the torus and the curve has two points, corresponding to 
$x = \pm 1$. It is quite straightforward that $\pi m(P_4) = 2D(i)$.

\subsubsection{A-polynomial of $m337$}

Using Culler's \verb|PE| tool, one can compute the A-polynomial for the manifold $m337$ in Snappy. It factorizes in two (almost identical) factors, each of bidegree $(20,13)$. Note that the curve $C$ defined by this $A$-polynomial is invariant by change of sign of any variable, by $(x,y)\to (1/x,1/y)$ and $(x,y)\to (\bar x,\bar y)$, so we prefer to work with the polynomial whose zero set consists of the points $(x^2,y^2)$, hereafter denoted by $P$. The multi-graph of the volume above $|x|=1$ (here $x$ is the eigenvalue of the meridian for snappy) is displayed in Figure \ref{fig:m337}. 

\begin{figure}[ht]
\begin{center}
\includegraphics[width = 8cm]{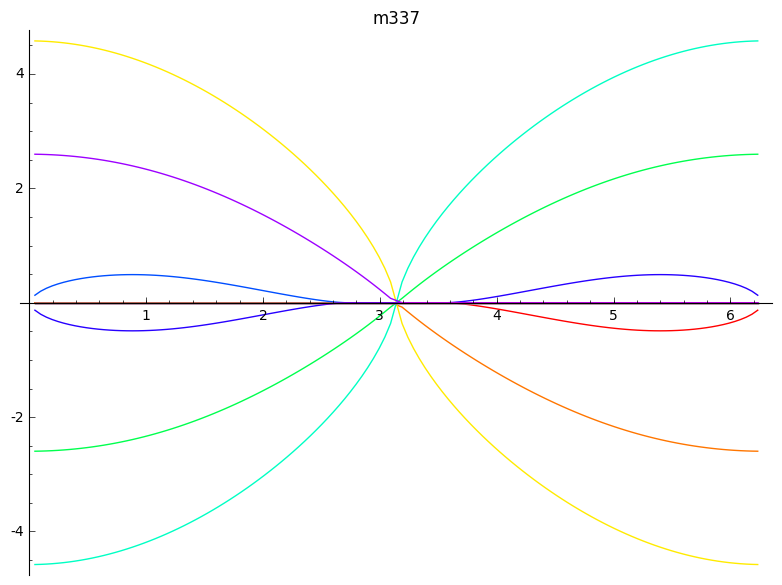}
\end{center}
\caption{Volume function for the A-polynomial of $m337$ above the circle $|x|=1$. The horizontal axis is $\arg(x)$.}\label{fig:m337}
\end{figure}

As is quite clear from the figure, the Mahler measure of this polynomial is 
(up to a factor $\pi$) of the form $v_1 + v_2 + 2v_3$ where $v_1$ and $v_2$ are the two 
positive volume above $x=1$ (i.e. the intersection with the vertical axis in the picture)
and $v_3$ is the common value of the two local maxima. 

The local maximum $v_3$ corresponds to an intersection between the curve $P=0$ and the torus, which does not lie above $x=1$. One can plot this intersection: it has a $1$-dimensional part and a $0$-dimensional part corresponding to the set of singular points of the algebraic curve that sit inside the torus, see Figure \ref{fig:m337-2}.

\begin{figure}[ht]
\begin{center}
\includegraphics[width = 8cm]{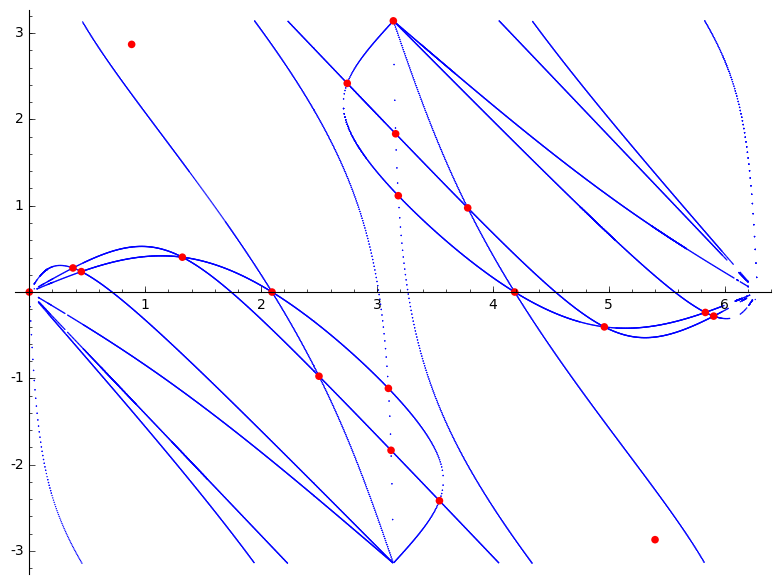}
\end{center}
\caption{Points on the torus (axis are the arguments of $x$ and $y$). The $1$-dimensional part is displayed in blue, the singular points in red.}\label{fig:m337-2}
\end{figure}

We find then $2$ singular points outside of the $1$-dimensional part, of coordinates $(x,y)$ where $x$ and $y$ are both algebraic numbers of degree $12$ in the same number field. They are not roots of unity and
the number field has 5 complex places. Note that the two values of $x$ are the points where the local maxima in Figure \ref{fig:m337} are attained. These two singular points are complex conjugate, so we may just study one of them. We can check, using for example SageMath, that 
two branches of $C$ goes through these points: they correspond to two distinct points in $\hat C$. One then check that the two slopes are not real and are complex conjugate: it explains the local maximum and the local minimum for the volume above this value of $x$. The other singular points inside the torus are easily seen not to contribute to the Mahler measure: at each of this point, there are two maximally tangent directions which are clearly real.

It is possible to express $v_1$, $v_2$ and $v_3$ as sums of dilogarithms of algebraic numbers, solving the gluing equations for $m337$. We do not give explicit details here.
An approximate value is:
$$\pi m(P) = v_1 + v_2 + 2v_3 = 8.1594511763 \pm 10^{-10}.$$ 
A detailed description of this computation will appear in an independent paper.

\subsection{An interpretation of the Mahler measure of the $A$-polynomial}
For $A\in \SL_2(\C)$ we denote by $||A||$ its spectral radius, that is the maximal modulus of an eigenvalue of $A$. 
\begin{definition}
Let $M$ be a closed oriented $3$-manifold and $K$ be a knot in $M$. Suppose that the character variety $X(M)$ is reduced of dimension 0. Then we set 
\[m(M,K)=\sum_{[\rho]\in X(M)} \log ||\rho(K)||\]
\end{definition}

Given an oriented manifold $M$ with toric boundary $S^1\times S^1$, we denote by $M_{p/q}$ the Dehn surgery with slope $p/q$ that is $M_{p/q}=M\cup_\phi D^2\times S^1$ where $\phi:\partial D^2\times S^1\to\partial M$ is given by $\phi(z,1)=(z^p,z^q)$. We denote by $K_{p/q}$ the knot $\{0\}\times S^1$ viewed in $M_{p/q}$.

\begin{proposition}\label{topinterpretation}
Suppose that $M$ is a manifold with $\partial M=S^1\times S^1$ with $A$-polynomial $A_M$ satisfying the following hypothesis:
\begin{enumerate}
\item The restriction map $r:X(M)\to X(\partial M)$ is birational on its image. 
\item The singular points of $X(M)$ do not restrict to torsion points in $X(\partial M)$. 
\end{enumerate}
Then we have 
\[\lim_{p^2+q^2\to\infty}m(M_{p/q},K_{p/q})=m(A)\]
\end{proposition}
\begin{proof}
A theorem of Dunfield states that the first assumption holds for any component of the character variety which contains a lift of a discrete and faithful representation (see \cite{Dun}) (notice that we assume implicitly that $X(M)$ is irreducible and reduced). This can be weakened as in \cite{le}. The roots of the Alexander polynomial of $M$ corresponds to singular points of $X(M)$, hence we assume that this polynomial does not vanish at roots of unity. These hypothesis are already present in \cite{marchemaurin} where they serve similar purposes. 

Let $\gamma_{p,q}\subset \partial M$ be the curve parametrised by $\gamma_{p/q}(z)=(z^p,z^q)$. As $\pi_1(M_{p/q})=\pi_1(M)/(\gamma_{p,q})$, a representation $\rho:\pi_1(M_{p/q})\to\SL_2(\C)$ corresponds to a representation $\rho:\pi_1(M)\to \SL_2(\C)$ such that $\rho(\gamma_{p/q})=1$. 
Denoting by $l,m$ the homotopy classes of $S^1\times\{1\}$ and $\{1\}\times S^1$ respectively, one can suppose that 
\[\rho(l)=\begin{pmatrix} x & * \\ 0 & x^{-1}\end{pmatrix},\quad \rho(m)=\begin{pmatrix} y & * \\ 0 & y^{-1}\end{pmatrix}\]
with $A(x,y)=0$ and $x^p y^q=1$. If $x\ne \pm 1$ or $y\ne \pm 1$, the pair $(x^{-1},y^{-1})$ satisfies the same equation and corresponds to the same representation up to conjugation. 
Reciprocally, to a solution $(x,y)$ of the equations $A(x,y)=0,x^py^q=1$ corresponds generically one representation by the birational assumption. By the second assumption, points where there are more solutions map to non-torsion points, and hence satisfy at most one equation of the form $x^py^q$. Hence, we can neglect them in the limit. 
We will see that the case when $x=\pm 1$ and $y=\pm 1$ does not contribute to the result hence we discard them also.

Now, the core of the torus $D^2\times S^1$ is mapped through $\phi$ to a curve of the form $K(t)=(t^r,t^s)$ where $ps-qr=1$. Parametrizing the solutions of $x^py^q=1$ by setting $x=t^{-q},y=t^{p}$ we find that the eigenvalue of $\rho(K_{p,q})$ is $t$. Hence 
\[m(M_{p/q})=\frac{1}{2}\sum_{t\ne 0,A(t^{-q},t^p)=0} |\log|t||=m(A(t^{-q},t^p))\]
where $m$ denotes the usual Mahler measure. As $p^2+q^2$ goes to infinity, the integral formula for the Mahler measure shows that this quantity converges to $m(A)$ and the conclusion follows. 
\end{proof}

\bibliographystyle{alpha}
\bibliography{biblio}

\end{document}